\documentclass{article}
\usepackage[12pt]{extsizes}
\usepackage[utf8]{inputenc} 
\usepackage[T1]{fontenc}    
\usepackage{hyperref}       
\usepackage{url}            
\usepackage{booktabs}       
\usepackage{amsfonts}       
\usepackage{nicefrac}       
\usepackage{graphicx}
\usepackage{xcolor}
\usepackage{amsrefs}
\usepackage{amsmath}
\usepackage{amsthm}
\usepackage{latexsym}
\usepackage{amssymb}
\usepackage[utf8]{inputenc}
\usepackage{mathtools}
\DeclareMathOperator{\per}{Per}
 \DeclareMathOperator{\pper}{PrePer}
 \DeclareMathOperator{\rad}{rad}
 \DeclareMathOperator{\res}{Res}
\newtheorem{theorem}{Theorem}
\newtheorem{lemma}{Lemma}
\newtheorem{corollary}{Corollary}
\newtheorem{conjecture}{Conjecture}
\newtheorem{example}{Example}[section]

\newcommand{\edv}{\mathrel\Vert}
\title{Rational Periodic Points of $x^d+c$ and Fermat-Catalan Equations}
\date{} 					

\author{Chatchawan Panraksa\thanks{This work is supported by Mahidol University International College via Research Grant No. 06/2018.} \\
	Applied Mathematics Program\\ Mahidol University International College\\
	999 Phutthamonthon 4 Road, Salaya\\
	Nakhonpathom, Thailand 73170\\
	\texttt{chatchawan.pan@mahidol.edu}
}

\hypersetup{
pdftitle={A template for the arxiv style},
pdfsubject={q-bio.NC, q-bio.QM},
pdfauthor={Chatchawan Panraksa},
}

\begin{document}
\maketitle
{\centering \it In memory of 
Sarat Sinlapavongsa (1979 - 2018)\par}
\begin{abstract}
	We study rational periodic points of polynomial $f_{d,c}(x)=x^d+c$ over the field of rational numbers, where $d$ is an integer greater than 2. For period 2, we classify all possible periodic points for degrees $d=4,6$. We also demonstrate the nonexistence of rational periodic points of exact period 2 for $d=2k$ such that $3\mid 2k-1$ and $k$ has a prime factor greater than 3. Moreover, assuming the $abc$-conjecture, we prove that $f_{d,c}$ has no rational periodic point of exact period greater than 1 for sufficiently large integer $d$ and $c\neq -1$.
\end{abstract}


\section{Introduction}

An arithmetic dynamical system is a pair $(\mathbb{P}^N(K),\phi)$ where $K$ is a number field and $\phi:\mathbb{P}^N(K)\longrightarrow \mathbb{P}^N(K)$ is a morphism on a projective space $\mathbb{P}^N(K)$ of dimension $N$. For $x\in \mathbb{P}^N(K)$, we denote $\phi^0(x)=x$ and $\phi^m(x)=\phi(\phi^{m-1}(x))$ where $m\in \mathbb{N}.$ We say a point $P\in \mathbb{P}^N(K)$ is a {\it periodic point} of $\phi$ if $\phi^{n}(P)=P$ for some $n\geq 1.$ If there is no positive integer $m<n$ such that $\phi^m(P)=P$, then $P$ is a periodic point of {\it exact period} $n.$ A point $Q\in\mathbb{P}^N(K)$ is called a {\it preperiodic point} of $\phi$ if $\phi^{n_1}(Q)=\phi^{n_2}(Q)$ for some positive integers $n_1\neq n_2.$ We denote $\per(\phi,\mathbb{P}^N(K))$ and $\pper(\phi,\mathbb{P}^N(K))$ the sets of periodic and preperiodic points of $\phi$ on $\mathbb{P}^N(K)$, respectively.
Morton and Silverman \cite{MR1264933} proposed a conjecture for morphisms over number fields : 
\begin{conjecture}[Morton-Silverman Uniform Boundedness Conjecture]
There exists a bound $B=B(D,N,d)$ such that if $K$ is a number field of degree $D,$ and $\phi:\mathbb{P}^N(K)\longrightarrow \mathbb{P}^N(K)$ is a morphism of degree $d\geq 2$ defined over $K,$ then the number of preperiodic points of $\phi$ is bounded by $B.$
\end{conjecture}
This conjecture is remarkably strong. For $(D,n,d)=(1,1,4),$ the conjecture implies that the size of the torsion subgroup of an elliptic curve is uniformly bounded (see \cite{MR488287}). This can be done via the associated Latt\`es map. Similarly, for $(D,n,d)=(D,1,4),$ the conjecture implies the Merel's theorem \cite{MR1369424}. Therefore, the set of associated Latt\`es maps of elliptic curves is a family of rational functions that is proved to satisfy the conjecture. Recently, under the assumption of the $abcd$-conjecture, Looper proved the Uniform Boundedness Conjecture by Morton and Silverman for the case of polynomials (see \cite{Looper1}).The conjecture for quadratic polynomials over $\mathbb{Q}$ was posed by Russo and Walde \cite{MR1270956} as follows: 
\begin{conjecture}[Uniform Boundedness Conjecture for $x^2+c$]\label{BoundedQuadratic}
If $n\geq 4,$ then there is no quadratic polynomial $f_c(x)=x^2+c\in\mathbb{Q}[x]$ with a rational point of exact period $n.$
\end{conjecture}
For exact period $n=1,2$ and 3, Walde and Russo \cite{MR1270956} described the periodic points of $f_c(x)=x^2+c$. Conjecture~\ref{BoundedQuadratic} has been proved for  $n=4$ and $n=5$ by Morton \cite {MR1665198}, and Flynn, Poonen and Shaefer  \cite{MR1480542}, respectively. For $n=6$, Stoll \cite{MR2465796} verified the conjecture under the assumption of the Birch and Swinnerton-Dyer conjecture for a certain genus 4 curve. Poonen \cite{PoonenClass} showed that Conjecture~\ref{BoundedQuadratic} implies $\#\pper(f_c,\mathbb{P}^1(\mathbb{Q}))\leq 9.$\\
The quadratic polynomial $f_c(x)=x^2+c$ belongs to a family of polynomials of the form $f_{d,c}(x)=x^d+c.$ Based on calculations, Hutz \cite{MR3266961} proposed an \textit{absolute} boundedness conjecture for $f_{d,c}(x)=x^d+c.$ We use ``absolute" to emphasize that the bound does not depend on the degree $d$.
\begin{conjecture}[Generalized Poonen]
For $n>3$, there is no integer $d\geq 2$ and $c\in \mathbb{Q}$ such that $f_{d,c}(x)=x^d+c$ has a rational periodic point of exact period $n$. For maps of the form $f_{d,c}$, we have  $$\#\pper(f_{d,c},\mathbb{P}^1(\mathbb{Q}))\leq 9.$$
\end{conjecture}
For odd degree $d$, Narkiewicz \cite{MR3447581} studied the periodic points of $f_{d,c}$ over number fields. In particular, his results yield the nonexistence of rational periodic points of exact period greater than one.
\begin{theorem}[Odd degree]\label{odd}
For $n>1$ and odd integer $d>2$, there is no $c\in\mathbb{Q}$ such that $f_{d,c}(x)=x^d+c$ has a rational periodic point of exact period $n$. Moreover, $$\#\pper(f_{d,c},\mathbb{P}^1(\mathbb{Q}))\leq 4.$$
\end{theorem}
It remains open for even degree $d>2$. Hutz \cite{MR3266961} also proposed the following conjecture.
\begin{conjecture}[Even degree]
For $n>2$ and even integer $d>2$, there is no $c\in \mathbb{Q}$ such that $f_{d,c}(x)=x^d+c$ has a rational periodic point of exact period $n$. Moreover, $$\#\pper(f_{d,c},\mathbb{P}^1(\mathbb{Q}))\leq 4.$$
\end{conjecture}
Doyle and Poonen \cite{DoylePoonen} showed that the number of preperiodic points of $z^d+c$ in a number field with a fixed eventual period is uniformly bounded.
Under the assumption of the $abcd$-conjecture, Looper \cite{Looper} showed that the uniform boundedness of $f_{d,c}(x)=x^d+c$ holds for number fields and one-dimensional function fields of characteristic zero. Looper's innovative techniques rely on global equidistribution statement for preperiodic points and adelic properties of preperiodic points. Here we state Looper's theorem only for the number field case (for details of the function field case and the $abcd$-conjecture, see \cite{Looper}).
\begin{theorem}[Uniform Boundedness for $x^d+c$]\label{Looper} Let $d\geq 2$ and $K$ be a number field. Let $f_{d,c}(x)=x^d+c\in K[x]$. If $d\geq 5,$ assume the $abc$-conjecture. If $2\leq d\leq 4$, assume the $abcd$-conjecture. There is a $B=B(d,K)$ such that $f_{d,c}$ has at most $B$ preperiodic points.
\end{theorem}
Note that if $K=\mathbb{Q},$ then the number of preperiodic points of $f_{d,c}(x)=x^d+c$ is bounded by a constant depends only on $d$.\\
This paper is structured as follows:\\
{\bf Section~\ref{period2_46}.} We analyze periodic points of exact period 2 of $f_{d,c}(x)=x^d+c$ for $d=4,6$ in detail.\\
{\bf Section~\ref{FCeq}.} We describe the connection between periodic points of exact period 2 of $f_{d,c}(x)=x^d+c$ and the Fermat-Catalan equations. We also show as an example that there are infinitely many $d$'s  such that $f_{d,c}$ has only trivial periodic points of exact period 2.\\
{\bf Section~\ref{abc}.} Under the assumption of the $abc$-conjecture, we show the existence of an absolute bound for the periodic points of $f_{d,c}(x)=x^d+c$ for sufficiently large $d$. In this section we prove the following theorem. The terms in the proof are similar to \cite{Looper} when $K=\mathbb{Q}$. Looper's strategy emphasizes on the detailed analysis of parameter $c$, that is, fixing $d$ and varying $c$.  The proof of the following theorem focuses on fixing $c$ and varying $d$.
\begin{theorem}\label{unifQ} Let $f_{d,c}(x)=x^d+c\in \mathbb{Q}[x]$. If the $abc$-conjecture is valid, then for sufficiently large degree $d$,
\begin{enumerate}
    \item $f_{d,c}$ has no rational periodic points of exact period greater than 1 except when $c=-1$,
    \item $\# \pper(f_{d,c},\mathbb{P}^1(\mathbb{Q}))\leq 4$.  
\end{enumerate}
 Moreover, if we assume the explicit $abc$-conjecture, then the result holds for $d\geq 7.$
\end{theorem}
For $K=\mathbb{Q}$, Theorem~\ref{Looper} and Theorem~\ref{unifQ} immediately imply the following theorem.
\begin{theorem}
Assume the $abcd$-conjecture. There is an absolute constant $B_0$ such that $\# \pper(f_{d,c},\mathbb{P}^1(\mathbb{Q}))\leq B_0,$ for all  $f_{d,c}(x)=x^d+c\in \mathbb{Q}[x]$ with $d\geq 2$.
\end{theorem}
\noindent {\bf Section~\ref{proofH81}.} The proof, provided by Andrew Bremner, describes the rational points on the hyperelliptic curve given by $H:y^2=x^5+81$. This result is used to analyze rational periodic points of exact period 2 of $f(x)=x^6+c$.\\
{\bf Section~\ref{prooflemmas}.}  The proofs of necessary lemmas for Theorem~\ref{unifQ} are provided.
\section{2-Periodic Points of $f_{d,c}$ for $d=4,6$}\label{period2_46}
For an even positive integer $d$, it is easy to see that $f_{d,0}(x)=x^d$ has rational periodic points of exact period 1, i.e., $0$ and $1$. Similarly, $f_{d,-1}(x)=x^d-1$ has rational periodic points of exact period 2, i.e., $0$ and $-1$. We will call these two cases, i.e., $c\in \{0,-1\}$, as {\bf trivial}. On the other hand, it is easy to see that if $f_{d,c}=x^d+c$ has a rational point $x\in\{0,1\}$ of period 1 or a set $\{0,-1\}$ of exact period 2, then $c=0$ or $c=-1$, respectively.\\ 
 For a polynomial $f(x)\in \mathbb{Q}[x]$, we define the $2^{nd}$ {\it dynatomic polynomial} $\Phi_2(x)$ as $$\Phi_2(x)=\dfrac{f^2(x)-x}{f(x)-x}.$$
It is easy to see that the roots of the $2^{nd}$ dynatomic polynomial are periodic points of period 2.
 For more details on dynatomic polynomials, see \cite{MR2316407}.
\\The following theorem describes all rational periodic points of exact period 2 for $f_{4,c}(x)=x^4+c$. 
\begin{theorem}\label{period4} 
There are infinitely many $c\in\mathbb{Q}$ such that $f_{4,c}(x)=x^4+c$ has rational periodic points of exact period 2. The parametrization of such $c$ and 2-periodic points $x_1,x_2$ are as follows. The parameter $c$ is in the form  $$\displaystyle c=\dfrac{t^6+4t^3-1}{4t^2}$$ and the 2-periodic points are
$$\displaystyle x_1,x_2=\dfrac{t^2\pm\sqrt{-t^4-2 t}}{2 t},$$ where $-t^4-2 t=y^2$ for some $y\in\mathbb{Q}$.
\end{theorem}
\begin{proof}
Consider the $2^{nd}$ dynatomic polynomial $\Phi_2(x,c)$ of $f_{4,c}(x).$
We have
\begin{align*}
    \Phi_2(x,c)&=\displaystyle\frac{f_{4,c}^2(x)-x}{f_{4,c}(x)-x}\\
    &=x^{12} + x^9 + 3 c x^8 + x^6 + 2 c x^5 + 3 c^2 x^4 + x^3 + c x^2 + 
 c^2 x + 1 + c^3.
\end{align*}
 To find the roots of $\Phi_2$, we solve for the zeroes of the resultant of $\Phi_2(x,c)$ and the trace of 2-periodic orbits $$T(t,c):=t-(x+f_{4,c}(x))$$ with respect to $x$. We have $$\res(\Phi_2,T)=(1 + 4 c t^2 - 4 t^3 - t^6)^2=0.$$ Thus, $$1 + 4 c t^2 - 4 t^3 - t^6=0,$$ or equivalently, $$\displaystyle c=\frac{t^6+4t^3-1}{4t^2}.$$  The $2^{nd}$ dynatomic polynomial satisfies  
\begin{align*}
0=\Phi_2(x,c)&=\Phi_2(x,\dfrac{t^6+4t^3-1}{4t^2})\\
&=\dfrac{\left(t^3-2 t^2 x+2 t x^2+1\right) \left(t^{15}+2 t^{14} x+\dots+8 t^2 x^4+2 t^2 x+2 t
   x^2-1\right)}{64 t^6}.
\end{align*}
We solve for rational solutions to the equation $$t^3-2 t^2 x+2 t x^2+1=0,$$ equivalently, $$x=\dfrac{t^2\pm\sqrt{-t^4-2 t}}{2 t}.$$ Therefore,  $(x,t)\in\mathbb{Q}^2$ if and only if there is $y\in \mathbb{Q}$ such that $y^2=-t^4-2 t$.
 Consider the elliptic curve $E: y^2=-t^4-2t$. We have that $E(\mathbb{Q})\cong \mathbb{Z}$. Therefore, there are infinitely many $c\in\mathbb{Q}$ such that $f_{4,c}$ has rational periodic points of exact period 2.\\
\end{proof}
\begin{example}
For $ t_0=-\dfrac{2}{5}$, the parameter $c$ and $2$-periodic points $x_1,x_2$ are \\$ c=\dfrac{t_0^6+4t_0^3-1}{4t_0^2}=-\dfrac{19561}{10000}$ and $x_1,x_2=\dfrac{t_0^2\pm\sqrt{-t_0^4-2 t_0}}{2 t_0}=-\dfrac{13}{10},\dfrac{9}{10}.$\\ For $f(x)=x^4-\dfrac{19561}{10000},$ we have $f(-\dfrac{13}{10})=\dfrac{9}{10}$ and $f(\dfrac{9}{10})=-\dfrac{13}{10}.$
\end{example}
Next, we show in Theorem~\ref{period6} that $f_{6,c}(x)=x^6+c$ has no rational points of exact period 2 for $c\in \mathbb{Q}\backslash \{-1\}$. In order to prove this theorem, we need to find the set of rational points on a hyperelliptic curve. The result is in the following theorem.
\begin{theorem}\label{H81}
Let $H: y^2=x^5+81.$ Then $H(\mathbb{Q})=\{(2,\pm 7),(0,\pm 9),(3,\pm 18),\infty\}.$
\end{theorem}
\begin{proof}
The proof is in Section \ref{proofH81}.
\end{proof}
\begin{theorem}\label{period6}
Let $f_{6,c}(x)=x^6+c$, where $c\in \mathbb{Q}\backslash \{-1\}$. Then $f_{6,c}$ has no rational point of exact period 2.  
\end{theorem}
\begin{proof}
We compute the resultant of dynatomic polynomial of $f_{6,c}(x)=x^6+c$ and the trace $t=x_1+x_2$. The resultant curve is transformed to $H: y^2=x^5+81.$ By Theorem~\ref{H81}, $$H(\mathbb{Q})=\{(2,\pm 7),(0,\pm 9),(3,\pm 18),\infty\}.$$ Then we check all corresponding points $H$ and values of $c.$ We have only $c=0,-1$ which are trivial.\\

\end{proof}
Theorem~\ref{period6} can also be proved using a result of a Fermat-Catalan equation $$x^3+y^3=z^5$$ (see \cite{MR1850605}).
We will prove more general cases via the Fermat-Catalan equations in the next section. The main result of Section 3 is that $f_{d,c}(x) = x^{2k} + c$ has no non-trivial 2-periodic points for infinitely
many values of $k$.
\section{2-Periodic Points of $f_{d,c}(x)=x^d+c$ and Fermat-Catalan Equations}\label{FCeq}
Walde and Russo \cite{MR1270956} gave the description of parameter $c$ and the corresponding periodic points of  $f_{2,c}(x)=x^2+c$ over $\mathbb{Q}$ . Narkiewicz \cite{MR2674537} and Sadek \cite{Sadek} extended the results to the family of functions $f_{d,c}(x)=x^d+c$ over a number field $K$. The following lemma can be proved by a similar argument to Lemma 7 in Hutz's article (see \cite{MR3266961}). His argument works for preperiodic points as well.
\begin{lemma}\cite{MR3266961}
Let $f_{d,c}(x)=x^d+c\in K[x]$ and let $\alpha\in \mathbb{A}(K)$ be a preperiodic point of $f_{d,c}$. Then for each nonarchimedean place $v$ such that $v(c)<0$, we have $v(c)=dv (\alpha)$.  For each nonarchimedean place $v$ such that $v(c)\geq 0$, we have $v (\alpha)\geq 0$.
\end{lemma}
The above lemma immediately implies the following corollary.
\begin{corollary}\label{ratform}
If $X/Z$ is a rational preperiodic point of $f_{d,c}(x)=x^d+c$ with $c=M/N$ and both $X/Z$ and $M/N$ are rational numbers expressed in the lowest terms with $Z$ and $N$ positive integers, then $N=Z^d.$
\end{corollary}
We will use Corollary~\ref{ratform} to show that the periodic points of exact period 2 of $f_{d,c}(x)=x^d+c$ can be described in the form of a Fermat-Catalan equation in the following theorem.
\begin{theorem}\label{FermatCatalan}
Let $k,n$ be positive integers such that $n>1$, and $f_{d,c}(x)=x^d+c\in \mathbb{Q}[x]$, where $d=2k$. If $X_1/Z$ and $X_2/Z=f_{d,c}(X_1/Z)$ are rational periodic points of exact period $n$ of $f_{d,c}$ with $c=C/Z^d$, and $X_1/Z, X_2/Z$ and $C/Z^d$ are rational numbers expressed in the lowest terms with integers $X_1,X_2$ and $Z$, then $\gcd(X_1,X_2)=1.$ Moreover, for $n=2$,
\begin{itemize}
    \item[$a)$] $X_1^k+X_2^k=\delta Z_1^{2k-1}$ for some $Z_1\in \mathbb{Z}$ and $\delta\in\{1,2\}$,
    \item[$b)$] if $k$ is odd, then $X_1^k+X_2^k= Z_1^{2k-1}$ for some $Z_1\in \mathbb{Z}.$
\end{itemize}
\end{theorem}
In the following proof, for integers $a, b$ and positive integer $n$, we write $a^n|| b$ if $a^n|b$ but $a^{n+1}\not\vert b$.
\begin{proof} Let $x_i=X_i/Z$ be rational periodic points of exact period $n$ of $f_{d,c}(x)$ for $i=1,\dots,n.$ Observe that 
$$\displaystyle \prod_{i=1}^{n}(x_i-x_{i+1})=\prod_{i=1}^{n}(f_{d,c}(x_i)-f_{d,c}(x_{i+1}))=\prod_{i=1}^{n}(x_i^{2k}-x_{i+1}^{2k}),$$
where $x_{n+1}=x_1$.
Thus, $\displaystyle \prod_{i=1}^{n}\dfrac{(x_i^{2k}-x_{i+1}^{2k})}{x_i-x_{i+1}}=1,$ or equivalently, $$\displaystyle \prod_{i=1}^{n}\dfrac{(X_i^{2k}-X_{i+1}^{2k})}{X_i-X_{i+1}}=Z^{n(2k-1)}.$$   Assume that there is $q\in \mathbb{N}$ such that $q\mid \gcd(X_1,X_2)$.\\
 Since $q\mid \gcd(X_1,X_2)$ and $$\displaystyle\prod_{i=1}^n(X_i^k+X_{i+1}^k)\frac{(X_i^k-X_{i+1}^k)}{X_i-X_{i+1}}=Z^{n(2k-1)},$$ we have $q\mid Z^{n(2k-1)}.$ However, $X_1$ and $Z$ are coprime. Thus, $q=1.$\\
 For period $n=2$, we have that $$\displaystyle(X_1^k+X_2^k)\frac{(X_1^k-X_2^k)}{X_1-X_2}=Z^{2k-1}.$$ Since $\gcd(X_1,X_2)=1$ and 
 $$\gcd(X_1^k+X_2^k,X_1^k-X_2^k)\mid\gcd(2X_1^k,2X_2^k),$$ we have that $\gcd(X_1^k+X_2^k,X_1^k-X_2^k)\mid 2.$\\ We consider 2 cases.\\
 \textbf{ Case 1. $\gcd(X_1^k+X_2^k,\dfrac{X_1^k-X_2^k}{X_1-X_2})=1$ }\\
 If $\gcd(X_1^k+X_2^k,\dfrac{X_1^k-X_2^k}{X_1-X_2})=1,$ then we easily have that $X_1^k+X_2^k=Z_1^{2k-1}$ for some integer $Z_1$ such that $Z_1\mid Z.$ \\ 
 \textbf{Case 2.} $\gcd(X_1^k+X_2^k,\dfrac{X_1^k-X_2^k}{X_1-X_2})=2$\\
 If $\gcd(X_1^k+X_2^k,\dfrac{X_1^k-X_2^k}{X_1-X_2})=2,$ then we have that $Z$ is even. Therefore, $X_1$ and $X_2$ are odd.\\
 If $4\mid X_1^k+X_2^k,$ then $2\edv X_1^k-X_2^k. $ Thus, $\dfrac{X_1^k-X_2^k}{X_1-X_2}$ is odd and $$\gcd(X_1^k+X_2^k,\dfrac{X_1^k-X_2^k}{X_1-X_2})=1,$$ a contradiction. Therefore, $2 \edv X_1^k+X_2^k.$ Thus, 
 $$X_1^k+X_2^k=2Z_1^{2k-1} \textrm{ and } \dfrac{X_1^k-X_2^k}{X_1-X_2}=2^{2k-2}Z_2^{2k-1}$$ for some integers $Z_1$ and $Z_2$ such that $\gcd(Z_1,Z_2)=1.$ \\
 If $k$ is an odd integer, then lifting the exponent lemma (see, for example, \cite{lte,titu}) implies that $\dfrac{X_1^k-X_2^k}{X_1-X_2}$ is odd. Thus, $$\gcd(X_1^k+X_2^k,\dfrac{X_1^k-X_2^k}{X_1-X_2})=1,$$ a contradiction.\\
From 2 cases, we can conclude that if $k$ is odd, then $X_1^k+X_2^k= Z_1^{2k-1}$ for some $Z_1\in \mathbb{Z}.$
\end{proof}
From the above, we can have a connection between the periodic points of exact period 2 of $f_{d,c}(x)=x^{2k}+c$ and the solutions to a corresponding ternary Diophantine equation   $$x^k+y^k=\delta z^{2k-1}$$ where $\delta\in\{1,2\}.$ To yield the result for the case when $k$ is an odd integer, it suffices to show the nonexistence of nontrivial solution to the Fermat-Catalan equation of signature $(k,k,2k-1)$ in the form $$x^k+y^k=z^{2k-1}.$$ For the case $k=3,$ we can yield the result  by employing the explicit Chabauty method as described by  Bruin \cite{MR1850605}. Note that we have also shown the nonexistence of nontrivial periodic points of exact period 2 of $f_{6,c}(x)=x^6+c$ via the Chabauty method in Theorem~\ref{H81}. The proof is more elementary than the proof in Bruin's paper since it deals directly with the periodic points of $f_{6,c}(x)$. The related Fermat-Catalan equations of signature $(n,n,p)$ as in, for example, \cite{MR1468926,MR2075481,MR2646760,MR3493372,MR1655978} are useful to study the periodic points of exact period 2 of $f_{d,c}.$ As an example, we can show the nonexistence of periodic points of exact period 2 of $f_{2k,c}(x)=x^{2k}+c$ when $3\mid 2k-1$ and $2k$ has a prime factor greater than 3. The following theorem is a partial result of Theorem 1.5 by Bennett, Vatsal and Yazdani in \cite{MR2098394}:
\begin{theorem}\label{BVY}\cite{MR2098394}
Let $C\in \{1, 2, 3, 5, 7, 11, 13, 15, 17, 19\}$ and prime number $n> \max\{C,4\}$, then the Diophantine equation
$x^n + y^n = Cz^3$ has no solutions in coprime nonzero integers $x, y$ and $z$ with $|xy| > 1$.
\end{theorem}
By Theorem~\ref{BVY} and Theorem~\ref{FermatCatalan} we can directly derive the nonexistence of periodic points of exact period 2 of $f_{2k,c}(x)=x^{2k}+c$ when $3\mid 2k-1$ and $k$ has a prime factor greater than 3 as in the following theorem since the corresponding Fermat-Catalan equation is $X^k+Y^k=\delta Z^{2k-1}$ where $\delta\in\{1,2\}$. 
\begin{theorem}
Let $k$ be an integer with a prime factor $p\geq 5$ and $f_{2k,c}(x)=x^{2k}+c$ where $c\in\mathbb{Q}\backslash \{-1\}$. If $3\mid 2k-1$, then $f_{2k,c}$ has no rational periodic point of exact period 2.
\end{theorem}
\section{$abc$-Conjecture and Absolute Boundedness of Periodic Points of $f_{d,c}$ over $\mathbb{Q}$}\label{abc}
In the previous section, we discuss the 2-periodic points. Even for this case, we can only treat $f_{d,c}(x)=x^d+c$ for some $d\in\mathbb{N}$. This section, we rely on the $abc$-conjecture proposed by Oesterlé and Masser (see \cite{MR754559,MR1065152,MR992208}) to yield results of the absolute boundedness of preperiodic points for $f_{d,c}(x).$   Let $n\geq 2$ be a positive integer such that $\displaystyle n=p_1^{r_l}p_2^{r_2}\cdots p_l^{r_l}$ and $p_i$'s are different primes. The {\it radical} of $n$ is defined by $$\rad(n)=p_1p_2\cdots p_l.$$
By convention, $\rad(1):=1.$ The $abc$-conjecture can be stated as follows.
\begin{conjecture} [\bf $abc-$ conjecture]
For every $\epsilon>0,$ there exists a constant $K_{\epsilon}$ such that for all triples $(a,b,c)$ of coprime integers, with $a+b=c$ such that
$$\max(|a|,|b|,|c|)<K_\epsilon \rad(|abc|)^{1+\epsilon}.$$
\end{conjecture}
The explicit version of the $abc$-conjecture was proposed by Baker \cite{MR2107944}. 
\begin{conjecture} [\bf Explicit $abc-$conjecture]\cite{MR2107944}\label{explicitabc_baker}
 There exists an absolute constant $K$ such that for all triples $(a,b,c)$ of coprime integers with $a+b=c$, $abc\neq 0$ and $N=\rad(|abc|),$ 
$$\max(|a|,|b|,|c|)<K N\dfrac{(\log(N))^\omega}{\omega!}$$

where $\omega$ is the total number of distinct primes dividing a, b and c.
\end{conjecture}
Laishram and Shorey \cite{MR2997579} proved that the Baker's version implies that for $\epsilon=\dfrac{3}{4}$, $K_{\epsilon}\leq 1$.
Later Chim, Shorey, and Sinha~\cite{ChimShoreySinha} improved the result in the following theorem.

\begin{theorem}\label{ChimShoreySinha}
 Assume the explicit $abc$-conjecture. For all triples $(a,b,c)$ of coprime integers with $a+b=c$, $abc\neq 0$ and $N=\rad(|abc|)$,
$$\displaystyle \max(|a|,|b|,|c|)<\min (N^{1.72},10N^{1.62991},32N^{1.6}).$$
\end{theorem}
The form of Theorem~\ref{ChimShoreySinha} is similar to Theorem 2.2 in \cite{ChimShoreyNair}.
For a polynomial $F(x)\in \mathbb{Z}[x]$ of degree greater than one, we can show that any periodic point of $F(x)$ has exact period at most 2. This is because its periodic points must be integers as well. Therefore, we can simply apply divisibility properties to yield the result. Hence, for $f_{d,c}(x)=x^d+c$ and rational points of period greater than 2, we can consider only the case $c\not\in\mathbb{Z}$.\\
We employ the $abc$-conjecture to analyze a Diophantine equation in the form, related to Fermat-Catalan equation,
$$\displaystyle x^d-y^d=\delta z^{d-1}$$

where $\delta$ is a ``\emph{small}'' variable compared to $x,y$ and $z$. The following lemma is used to prove the nonexistence of rational periodic points of $f_{d,c}$ for sufficiently large integer $d$.
\begin{lemma}\label{abclemma}
Assume the $abc-$conjecture. Then, for each $\epsilon>0$, there is $K_{\epsilon}$ such that the system
\begin{align*}
\begin{cases}
X_2^d-X_1^d=(X_3-X_2)Z^{d-1} \neq 0\\
\gcd(X_1,X_2)=1\\
\max(|X_1|,|X_2|,|X_3|)=|X_2|\\
W=\max(|X_1|,|X_2|,|X_3|,|Z|)>1
\end{cases}
\end{align*}
satisfies $d<\log_b(2^{1+\epsilon}K_{\epsilon})+5+4\epsilon$, for all $b\in (1,W]$.
Therefore, the system has no integer solutions for sufficiently large positive integer $d.$\\
Moreover, if we assume the explicit $abc$-conjecture, then the system has no integer solutions for $d\geq 10.$
\end{lemma}
\begin{proof} 
Let $\epsilon>0.$ By the $abc$-conjecture, there is a constant $K_{\epsilon}>0$ such that 
$$M:=\max(|X_2|^d,|X_1|^d,|(X_3-X_2)Z^{d-1}|)<K_{\epsilon} \rad(|X_2X_1(X_3-X_2)Z|)^{1+\epsilon}
.$$
We consider 2 cases.\\
{\bf Case 1.} $|X_2|\leq |Z|.$\\
We have
\begin{align*}
|(X_3-X_2)Z^{d-1}|\leq M &<K_{\epsilon}\rad(|X_2X_1(X_3-X_2)Z|)^{1+\epsilon}\\
                                      &\leq K_{\epsilon}(|Z\cdot Z\cdot 2X_2\cdot Z|)^{1+\epsilon}\\
&\leq 2^{1+\epsilon}K_{\epsilon}(|Z\cdot Z\cdot Z\cdot Z|)^{1+\epsilon}
\\
&\leq 2^{1+\epsilon} K_{\epsilon}|Z|^{4+4\epsilon}.
\end{align*}
Thus, $|Z|^{d-1}\leq |X_3-X_2||Z|^{m}<2^{1+\epsilon}K_{\epsilon}|Z|^{4+4\epsilon}.$ Since  $|Z|>1$, we have $$d-1<\log_{|Z|}(2^{1+\epsilon}K_{\epsilon})+4+4\epsilon.$$
{\bf Case 2.} $|X_2|> |Z|.$\\
We have
\begin{align*}
|X_2|^d\leq M &<K_{\epsilon} \rad(|X_2X_1(X_3-X_2)Z|)^{1+\epsilon}\\
                                      &\leq K_{\epsilon}(|X_2\cdot X_2\cdot 2X_2\cdot X_2|)^{1+\epsilon}\\
&\leq 2^{1+\epsilon}K_{\epsilon}|X_2|^{4+4\epsilon}.
\end{align*}
Thus, $|X_2|^d<2^{1+\epsilon}K_{\epsilon}|X_2|^{4+4\epsilon}.$ Since $|X_2|>1$, we have $$d<\log_{|X_2|}(2^{1+\epsilon}K_{\epsilon})+4+4\epsilon.$$
Let $W=\max{(|X_1|,|X_2|,|X_3|,|Z|)}$. From both cases, we can conclude that $$d-1<\log_{W}(2^{1+\epsilon}K_{\epsilon})+4+4\epsilon.$$
This implies $d<\log_b(2^{1+\epsilon}K_{\epsilon})+5+4\epsilon$, for all $b\in (1,W]$.\\
Thus, for $b=2$, the system has no integer solutions for $d\geq \log_2(2^{1+\epsilon}K_{\epsilon})+5+4\epsilon.$\\
Now assume the explicit $abc$-conjecture.\\
Theorem~\ref{ChimShoreySinha} implies that for $\epsilon=0.72$, we can choose $K_{\epsilon}=1$. If $$d\geq 10\geq \log_2(2^{1.72})+7.88=9.6,$$ then the system has no integer solutions.

\end{proof}
Due to the result of odd degree in Theorem~\ref{odd}, we will focus on even degree $d$. Lemma \ref{abclemma} is used to analyze periodic and preperiodic points of $f_{d,c}(x^d)=x^d+c$.
The following lemma is an observation for the relation between different fixed points and points of exact period 2.

\begin{lemma}\label{diff_fixed}
Let $d$ be a positive even integer. Then $f_{d,c}(x)=x^d+c\in\mathbb{C}[x]$ has different fixed points $x_1,x_2\in\mathbb{C}$ if and only if $f_{d,c'}(x)=x^d+c'$ has points of exact period 2, namely, $y_1=-x_1$ and $y_2=-x_2$, where $c'=-x_1-x_2^d.$
\end{lemma}
\begin{proof}
Assume that $x_1$ and $x_2$ are different fixed points of $f_{d,c}(x)=x^d+c$. Therefore, $x_1^d+c=x_1$ and $x_2^d+c=x_2.$ Equivalently, $x_1-x_1^d=c=x_2-x_2^d.$ This implies $x_1+x_2^d=x_2+x_1^d$. Let $y_1=-x_1$ ans $y_2=-x_2.$ Thus, $y_1-y_2^d=y_2-y_1^d.$ Choose  $c'=y_1-y_2^d\,\,(=y_2-y_1^d).$ Since $d$ is even, $y_2^d+c'=y_1$ and $y_1^d+c'=y_2.$  \\
Conversely, assume that $y_1 $ and $y_2$ are different periodic points of exact period 2 of $f_{d,c'}=x^d+c'$. We can choose $x_1=-y_1$ and $x_2=-y_2$. Then $x_1$ and $x_2$ are different fixed points of $f_{d,c}$, where $c=x_1-x_1^d\,\, (=x_2-x_2^d)$.
\end{proof}
The above lemma is used to prove a certain case of Lemma \ref{fixedpoint}.
\begin{lemma}\label{fixedpoint}
Let $c\in \mathbb{Q}\backslash \{0\}$. Assume the $abc$-conjecture. Then $f_{d,c}(x)=x^d+c$ has at most one rational fixed point for sufficiently large degree $d$. Moreover, if we assume the explicit $abc$-conjecture, then the result holds for $d\geq 8.$
\end{lemma}
\begin{proof}
The proof is in Section \ref{prooflemmas}.
\end{proof}

 We borrow notation of types of preperiodic points from \cite{PoonenClass}. For a preperiodic point $x_0$ of a function $f(x)$, we say that it is a point of {\it type} $n_m$ if $   m$ is the smallest natural integer (including $0$) such that $f^m(x_0)$ is a point of exact period $n$. Note that if $x$ is a preperiodic point of $1_1$-type, then $f_{d,c}(f_{d,c}(x))=f_{d,c}(x)$ but $f_{d,c}(x)\neq x$. For sufficiently large degree $d$, by Lemma \ref{fixedpoint}, $f_{d,c}$ has at most one fixed point, i.e., $f_{d,c}(x)=-x$. We will show that there is no point of type $1_2$ in the following lemma.
\begin{lemma}\label{no12}
Let $f_{d,c}(x)=x^d+c\in \mathbb{Q}[x]$. Assume the $abc$-conjecture.  For sufficiently large $d$, $f_{d,c}$ has no point of type  $1_2$. Moreover, if we assume the explicit $abc-$conjecture, then the result holds for $d\geq 8.$
\end{lemma}
\begin{proof}
The proof is in Section~\ref{prooflemmas}.
\end{proof}
We will use Lemmas \ref{abclemma}, \ref{fixedpoint} and \ref{no12} to prove Theorem~\ref{unifQ} that the $f_{d,c}$ has no $n$-rational periodic points for $n>1$ and sufficiently large degree $d$. Hence the absolute boundedness for the rational preperiodic points of $f_{d,c}$ holds for sufficiently large degree $d$ under the assumption of $abc$-conjecture.

\begin{proof}( Theorem~\ref{unifQ})
Let $n$ be an integer greater than 1. Let $x_i$'s be periodic points of exact period $n$ such that $f_{d,c}(x_i)=x_{i+1}$ and $f_{d,c}(x_n)=x_{n+1}=x_1$ for $i=1,2,\dots,n.$ By Theorem~\ref{odd}, we can assume that $d$ is even.

Without loss of generality, assume that $\displaystyle |x_2|=\max_{1\leq i\leq n}(|x_i|).$
By Corollary \ref{ratform} and Theorem~\ref{FermatCatalan}, we can write $x_i=\dfrac{X_i}{Z}$ such that  $\gcd(X_i,X_{i+1})=1=\gcd(X_i,Z)$ for $i=1,2,\dots,n.$
We have that $f_{d,c}(x_2)-f_{d,c}(x_1)=x_3-x_2$ or, equivalently, $x_2^d-x_1^d=x_3-x_2.$ Thus, $\dfrac{X_2^d-X_1^d}{X_3-X_2}=Z^{d-1}$, i.e.,
\begin{equation}\label{maineq}
    X_1^d-X_2^d=(X_3-X_2)Z^{d-1}.
\end{equation}
 If $\max(|X_2|,|Z|)=1,$ then Corollary \ref{ratform} implies that $c\in \mathbb{Z}$. Thus,  $f_{d,c}(x)=x^d+c\in\mathbb{Z}[x]$ and $X_1=x_1,X_2=x_2,X_3=x_3\in\{0,1,-1\}$. Since $f_{d,c}(x)\in\mathbb{Z}[x]$, we have $n\leq 2$ and $c\in\{0,-1,-2\}$. Otherwise, $f_{d,c}(a)$ converges to $\infty$ for every $a\in \mathbb{R}$. If $c=0$ or $-2$, we can only have fixed points $x\in \{0,1\}$ or $x=-1$, respectively. Thus, period $n=1,$ a contradiction. If $c=-1$ and $d$ even, then period $n=2$. This is also a trivial case. \\
For the case $\max(|X_2|,|Z|)>1,$ we assume the $abc$-conjecture. From Eq.~\ref{maineq}, and properties of $X_1,X_2,X_3$ and $ Z$, we have the following system:
\begin{align}\label{abcsystem}
\begin{cases}
X_2^d-X_1^d=(X_3-X_2)Z^{d-1}\neq 0\\
\gcd(X_1,X_2)=1\\
\max(|X_1|,|X_2|,|X_3|)=|X_2|\\
\max(|X_2|,|Z|)>1.\\
\end{cases}
\end{align}
The desired result follows directly from Lemma \ref{abclemma} for $k=d=l$ and $m=d-1$.
Therefore, if $d$ is sufficiently large and $c\neq -1$, then $f_{d,c}$ cannot have a rational point of period greater than 1. Therefore, $f_{d,c}$ can have at most 4 fixed points (with no preperiodic points) if $d$ is odd, and at most 2 fixed points if $d$ is even (including $\infty$). By Lemma \ref{no12}, $\#\pper(f_{d,c},\mathbb{P}^1(\mathbb{Q}))\leq 4$. For $c= -1$ and $d$ even, it is easy to see that $f_{d,-1}$ can only have points of exact period 2: $x_1=0,x_2=-1$ with the only one preperiodic point $x_0=1$, and $\infty$ as a fixed point.\\
For the rest of the proof, we assume explicit $abc$-conjecture.
By Lemma \ref{abclemma}, the result holds for $d\geq 10.$\\
Consider Eq.~\ref{abcsystem} for the case $d= 8$. \\
 Let $W=\max(|X_1|,|X_2|,|X_3|,|Z|)$. Theorem~\ref{ChimShoreySinha} implies that for $\epsilon=0.62991$, we can choose $K=10$ so that, by Lemma \ref{abclemma}, $$d<\log_b(2^{1+\epsilon}K_{\epsilon})+5+4\epsilon \textrm{ for } b\in (1,W].$$ If $W>1300$, then $$8<\log_{1300}(2^{1+0.62991}(10))+5+4(0.62991)<7.999,$$ a contradiction.
\\ If $W\leq 1300$ and period $n=2$, then we can exhaust all possible values of $X_1,X_2,X_3$ and $Z$. The only possible value of $c$ is $c=-1$  \footnote{The computation was done on MS Windows 10 running Julia 1.4.1 and AMD Ryzen 7 2700X 3.7 Ghz processor and 16 Gb RAM.}. For $W\leq 1300$ and period $n>2$, it was examined by Hutz \cite{MR3266961} in table 1 with the height bound $1000000$ that, for $3\leq d\leq 11$, $f_{d,c}(x)=x^d+c$ has no periodic points with exact period greater than 2.\\
For odd $d$, by Theorem~\ref{odd}, $f_{d,c}$ has at most 4 preperiodic (fixed) points.
\end{proof}
\section{Proof of Theorem~\ref{H81} by Andrew Bremner}\label{proofH81}

The following Theorem is used to prove Theorem~\ref{period6} that the rational periodic points of period 2 for $f_{6,c}(x)=x^6+c$ are trivial.  The similar hyperellipic curve, $y^2=x^5-7$, is studied in \cite{MR2455707} by Bremner as an example. Both $y^2=x^5-7$ and $y^2=x^5+81$ have Jacobians of rank 2. The MAGMA code for the computation in the proof is available at  \cite{bremner_magma}.

\begin{proof} ( Theorem~\ref{H81})
Let $K=\mathbb{Q}(\theta), \theta^5=3,$ with the real value $\theta_0\sim 1.24573$. Let $O_K$ be the ring of integers of $K$. The class number of $K$ is 1, and there are two units, normalized to have norm $+1,$ which we may take as
$$\epsilon_1=\theta^3+\theta^2-\theta-2,\,\,\,\,\epsilon_2=\theta^4+2\theta+4.$$
We have the ideal factorizations
$$(3)=(\theta)^5,\,\,\,\,\,(5)=(p_5)^5-(\theta^3+\theta^2-1)^5,\,\,\,\, \text{ where } \theta\equiv -2\mod p_5.$$
Let $\phi$ be the evaluation homomorphism $\theta\rightarrow \theta_0.$ Then
$$\phi(e_1)\sim 0.239297,\,\,\,\, \phi(e_2)\sim 8.89969.$$
We set $(u,v)=(x/z^2,y/z^5),$ with $(x,z)=1,$ and the equation becomes
$$y^2=x^5+81z^{10}=(x+\theta^4z^2)(x^4-\theta^4x^3z^2+\theta^8x^2z^4-\theta^{12}xz^6+\theta^{16}z^8).$$
A common (ideal) divisor of the two factors on the left divides $(x+\theta^4z^2)$ and $(5\theta^{16}),$ so can only contain factors $(\theta)$ and $(p_5)$. Suppose $p_5$ is a common divisor. Then $x+16z^2\equiv 0\mod 5,$ so $x+z^2=5M,$ say, with
$$y^2=x^5+81z^{10}=(-z^2+5M)^5+81z^{10}\equiv 5z^{10}\mod 25,$$
forcing $5\mid y,5\mid z, 5\mid x,$ a contradiction. Thus, any common divisor can only be a power of $(\theta).$ Further, when $(\theta)$ is a common divisor, then $\theta\mid x,$ so that $3\mid x,$ which implies that the greatest common divisor is precisely $(\theta)^4.$ In summary, the greatest common divisor of the two factors is $(g),$ where $g=1,\theta^4.$\\
We thus deduce element equations
\begin{align}
x+\theta^4z^2 &=gua^2,\\
x^4-\theta^4x^3z^2+\theta^8x^2z^4-\theta^{12}xz^6+\theta^{16}z^8 &=gu^{-1}b^2
\end{align}
for some unit $u,$ and elements $a,b\in O_K$ with $gab=y.$ Without loss of generality, $u=\pm \epsilon_1^{i_1}\epsilon_2^{i_2},$ where $i_1,i_2=0,1.$ The first equation above implies $\phi(g)\phi(u)\phi(a)^2=x+\theta_0^4z^2>0,$ so we must have $\phi(u)>0,$ and the negative sign is impossible.\\
It suffices to find all $x,z\in \mathbb{Z}$ satisfying $(2)$; and since $g=1,\theta^4$ is a perfect square, it suffices to consider solely the case $g=1.$ Let
$$C_u: X^4-\theta^4X^3+\theta^8X^2-\theta^{12}X+\theta^{16}=u^{-1}b^2,$$
for $u=1,\epsilon_1,\epsilon_2,\epsilon_1\epsilon_2.$\\
{\bf Case $u=1.$} The curve $C_1$ is birationally equivalent to the curve
$$E_1: y^2 = x^3 - 5 x^2 + 5 x.$$
The elliptic Chabauty technique developed by Bruin and implemented into Magma is able to treat this curve without problem (the auxiliary prime used being 7), and the result is that the only points $X=x/z^2$ on $C_1$ for $x,y\in\mathbb{Z},$ occur for $(x,z)=(1,0),(0,1),$ corresponding to the point at infinity, and $(u,v)=(0,9)$ on the original curve.\\
\textbf{Case $u=\epsilon_1.$} The curve $C_{\epsilon_1} $ contains the point
$$(-2,\theta^4-2\theta^3-3\theta^2-\theta+2)$$
allowing construction of a birational map to the curve
$$E_{\epsilon_1}: y^2=x^3+(-2\theta^3+3\theta^2-4\theta+4)x^2+(5\theta^4-8\theta^3+8\theta^2-4\theta-4)x.$$
Again, elliptic Chabauty techniques using Magma, with auxiliary prime 11, show that the only point $X=x/z^2$ on $C_{\epsilon_1}$ for $x,z\in \mathbb{Z},$ is given by $(x,z)=(-2,1),$ corresponding to $(u,v)=(-2,7)$ on the original curve.\\
\textbf{Case $u=\epsilon_2.$} The point on $C_{\epsilon_2}$
$$(-2\theta,-4\theta^4-4\theta^3-4\theta^2-6\theta-9)$$
allows construction of a birational map to the curve
$$E_{\epsilon_2}: y^2=x^3+(-\theta^4-\theta^2-\theta-2)x^2+(\theta^4+\theta^3+\theta^2+2\theta+2)x.$$
Elliptic Chabauty techniques using Magma, with auxiliary prime 11, show that there are no points $X/z^2$ on $C_{\epsilon_2}$ for $x,z\in\mathbb{Z}.$\\
\textbf{Case $u=\epsilon_1 \epsilon_2.$} The point on $C_{\epsilon_1\epsilon_2}$
$$(\frac{1}{4}(\theta^4+3\theta^3-3\theta^2+3\theta-3),\,\,\,\,\frac{1}{16}(-51\theta^4+39\theta^3+117\theta^2+63\theta-99))$$
leads to a birational map to the curve
$$E_{\epsilon_1\epsilon_2}: y^2=x^3+(-2\theta^3-2\theta^2+\theta+4)x^2+(-4\theta^3-3\theta^2+4\theta+8)x.$$
Elliptic Chabauty techniques in Magma, with auxiliary prime 11, show that the only points $X=x/z^2$ on $C_{\epsilon_1\epsilon_2}$ for $x,z\in \mathbb{Z},$ are given by $(x,z)=(-1,1),(3,1).$ The former is ``\emph spurious" in that equation $(1)$ is not satisfied for $(x,z)=(-1,1);$ and the latter corresponds to $(u,v)=(3,18)$ on the original curve.

\end{proof}
\section{Proofs of Lemmas \ref{fixedpoint} and \ref{no12}}\label{prooflemmas}

\begin{proof} ( Lemma \ref{fixedpoint})
Let $c=\dfrac{C}{Z^d}\in \mathbb{Q}.$ For $c\in \{-2,-1,0,1,2\}$, we can verify directly that only $f_{d,0}(x)=x^d$ has 2 rational fixed points $0$ and $1$ (with only one preperiodic point $-1$) for even $d\geq 4$, and 3 rational fixed points $-1,0,1$ for odd $d\geq 3$. For every other $c$'s, $f_{d,c}$ has at most 1 rational fixed point.\\
For $c\in\mathbb{Q}\backslash\{-2,1,0,1,2\}$, we have that $0$ is not a preperiodic point of $f_{d,c}(x)=x^d+c$. Otherwise, if $c\in \mathbb{Q}\backslash \mathbb{Z}$, then for some prime 
$p$, $f_{d,c}^n(0)$ diverges to infinity $p-$adically. For $c\in \mathbb{Z}\backslash \{-2,1,0,1,2\},$ $f_{d,c}^n(0)$ diverges to infinity.\\  If $c \in\mathbb{Z}\backslash \{-2,1,0,1,2\}$, then $f_{d,c}(x)$ has at most one fixed point. Otherwise, if we have 2 integral fixed points $X\neq Y$, then $X^d-X=c=Y^d-Y$, equivalently, $\dfrac{X^d-Y^d}{X-Y}=1$. Note that $\left| \dfrac{X^d-Y^d}{X-Y}\right|\geq 2$, for $\max(|X|,|Y|)\geq 2$ and $d\geq 3$. Since $\max(|X|,|Y|)\leq 1$, we have $X,Y\in\{-1,1\}$ and $d$ is odd. Therefore, $c=\pm 2$, a contradiction. \\ Consider the case $c\not\in \mathbb{Z}$, i.e., $|Z|>1.$
Assume that there are 2 fixed points $$x_1=\dfrac{X_1}{Z}\neq x_2=\dfrac{X_2}{Z},$$ where $\dfrac{X_1}{Z}$ and $\dfrac{X_2}{Z}$ are written in the lowest forms. Without loss of generality, assume $|X_1|\geq |X_2|$. Since $x_1$ and $x_2$ are fixed points, $$x_1-x_1^d=c=x_2-x_2^d.$$
Therefore, we have $$X_1^d-X_2^d=(X_1-X_2)Z^{d-1}\neq 0.$$ If $\gcd(X_1,X_2)\neq 1$, then $\gcd(X_1,Z)>1$, a contradiction.\\
Thus, $\gcd(X_1,X_2)=1$ and $\max(|X_1|,|X_2|,|Z|)>1$. By Lemma \ref{abclemma}, the equation has no integer solution for sufficiently large $d$. Moreover, if we assume the explicit $abc-$conjecture, then it is true for $d\geq 10.$\\
Consider for the case $8\leq d<10$.\\
For $d=8$, by Lemma \ref{diff_fixed}, there is $f_{8,c'}(x)=x^8+c'\in\mathbb{Q}[x]$ with 2-periodic points, say, $y_1=-x_1=\dfrac{-X_1}{Z}$ and $y_2=-x_2=\dfrac{-X_2}{Z}.$ Since $$f_{8,c'}(y_1)-y_2=-c'=f_{8,c'}(y_2)-y_1,$$ we have $X_1^8-X_2^8=(X_1-X_2)Z^7.$\\
Theorem~\ref{ChimShoreySinha} implies that for $\epsilon=0.62991$, we can choose $K_{\epsilon}=10$.\\ Let $W=\max(|X_1|,|X_2|,|Z|)$.  Applying the same argument in Lemma \ref{abclemma}, we have $$|W|^7<2^{1+\epsilon}K_{\epsilon}W^{4+4\epsilon}$$ or equivalently, $$|W|^{3-4\epsilon}<2^{1+\epsilon}K_{\epsilon}.$$ If $W>1300$, then
$$31<1300^{0.48036}<W^{0.48036}=W^{3-4(0.62991)}<2^{1.62991}(10)<30.95<31,$$ a contradiction.\\ If $W\leq 1300$, then we can exhaust all possible values of $X_1,X_2$, and $Z$. None of the values can be points of period 2, except $c'=-1$. By Lemma \ref{diff_fixed}, this yields the value of $c=0,$ a contradiction. \\
\end{proof}

\begin{proof}( Lemma \ref{no12})
Assume to the contrary that there is $r\in \mathbb{Q}$ such that $r$ is a $1_2$-type point of $f_{d,c}$. Thus, $f_{d,c}(r)=-x_0$ for a fixed point $x_0\in \mathbb{Q}$. Since $c=x_0-x_0^d,$ we have $$f_{d,c}(r)=r^d+x_0-x_0^d=-x_0.$$ Thus, $r^d+2x_0-x_0^d=0.$ By Corollary~\ref{ratform}, write $r=\dfrac{R}{Z}$ and $x_0=\dfrac{X_0}{Z}$ in the lowest forms. Thus, $\gcd(X_0,Z)=1=\gcd(R,Z).$ Then
\begin{align}
    \dfrac{R^d}{Z^d}+\dfrac{2X_0}{Z}-\dfrac{X_0^d}{Z^d} &=0 \nonumber\\
    X_0^d-R^d&=2X_0Z^{d-1} \nonumber\\
    X_0^{d-1}-\dfrac{R^d}{X_0}&=2Z^{d-1}.\label{abc_prefix}
\end{align}
Since $X_0$ and $Z$ are integers, $X_0\mid R^d$. We consider 2 cases.\\
\textbf{Case 1.} $X_0$ is odd.\\
Then $R$ is odd (otherwise, $2Z^{d-1}$ is odd).\\
Since $$\gcd(R,2Z)=1=\gcd(X_0,2Z)=\gcd(X_0^{d-1},2Z^{d-1}),$$ we have $\gcd(X_0^{d-1},\dfrac{R^d}{X_0})=1.$ Thus, $ X_0^{d-1},\dfrac{R^d}{X_0}$ and $ 2Z^{d-1}$ are coprime. Assume the $abc$-conjecture with $\epsilon =1$. Then there is a constant $K_1>0 $ such that\\
\begin{align*}
 \max(|X_0^{d-1}|,|\dfrac{R^d}{X_0}|, |2Z^{d-1}|)&< K_1\rad(|X_0^{d-1}\cdot \dfrac{R^d}{X_0}\cdot 2Z^{d-1}|)^2\\
 &\leq K_1\rad(|2X_0RZ|)^2.
\end{align*}
If $|X_0|=\max(|X_0|,|R|,|Z|)>1$, then \\
\begin{align*}
    |X_0^{d-1}|&\leq  \max(|X_0^{d-1}|,|\dfrac{R^d}{X_0}|, |2Z^{d-1}|)\\
       &<K_1\rad(|2X_0RZ|)^2\\
       &\leq K_1\rad(|2X_0^3|)^2.
\end{align*}
\begin{equation*}
        \text{Thus,}\, |X_0^{d-1}|< K_14X_0^6.
\end{equation*}
Therefore, $|X_0^{d-7}|< 4K_1.$
Thus, for sufficiently large $d$, there are no integers solution to Eq.~\ref{abc_prefix}. Similarly, we can yield the same result when $|R|$ and $|Z|$ are the maximum values. Moreover, if we assume the explicit $abc$-conjecture with $\epsilon=1$ and $K_1<1$, the result holds when $d\geq 9.$ For $d=8,$ we check all values $|X_0|,|R|,|Z|\leq 3$. None of the values, except $X_0=R=0$, yields the solution to the equation.\\
\textbf{Case 2.} $X_0$ is even.\\
Since $X_0\mid R^d$, $R$ is even. Consider $$\dfrac{X_0^{d-1}}{2}-\dfrac{R^d}{2X_0}=Z^{d-1}.$$ Since $\dfrac{X_0^{d-1}}{2}$ is an integer, so is $\dfrac{R^d}{2X_0}.$ Since $\gcd(X_0,Z)=1$, $\gcd(\dfrac{X_0^{d-1}}{2},\dfrac{R^d}{2X_0})=1.$ From the previous argument in Case 1 for the tuple $(\dfrac{X_0^{d-1}}{2},\dfrac{R^d}{2X_0},Z^{d-1})$, we can yield the same result in this case.\\
\end{proof}

%

{\bf Acknowledgments}\\
I am deeply grateful to Professor Marc Hindry, Professor Loïc Merel, Professor Michel Waldschmidt,  and Professor Lawrence C. Washington for guiding me through this work. I am immensely thankful to Professor Andrew Bremner for the proof of Theorem~\ref{H81}. I would like to thank Professor Benjamin Hutz, Professor Nicole Looper, and Professor W\l adys\l aw Narkiewicz for helpful discussions. I appreciate Professor Pierre Charollois, Professor Alain Kraus, Professor Shanta Laishram,  and Professor Aram Tangboonduangjit for the helps. I gratefully thank the referee for  giving such constructive comments. I am greatly thankful to Dr. Patcharee Panraksa for the inspiration. The majority of the work was completed at the Institut de Mathématiques de Jussieu-Paris Rive Gauche with leave from the Mahidol University International College from Research Grant No. 06/2018. I wish to express my sincere gratitude to these institutions for their facilities and supports. 


\begin{thebibliography}{1}
\bib{titu}{book}{
	author={Andreescu, T.},
	author={Dospinescu, G.},
	title={$ Problems~ From~ the~ Book$, 2nd Edition},
	publisher={XYZ Press},
	date={2010},
	ISBN={978-0979926907},
}

\bib{MR2107944}{article}{
	author={Baker, A.},
	title={ Experiments on the {$abc$}-conjecture},
	date={2004},
	ISSN={0033-3883},
	journal={$ Publ.~ Math.~ Debrecen$},
	volume={65},
	number={3-4},
	pages={253\ndash 260},
}

\bib{MR2646760}{article}{
	author={Bennett, M.~A.},
	author={Ellenberg, J.~S.},
	author={Ng, N.~C.},
	title={The {D}iophantine equation {$A^4+2^\delta B^2=C^n$}},
	date={2010},
	ISSN={1793-0421},
	journal={$Int. ~J.~ Number ~Theory$},
	volume={6},
	number={2},
	pages={311\ndash 338},
	url={https://doi.org/10.1142/S1793042110002971},
}

\bib{MR2098394}{article}{
	author={Bennett, M.~A.},
	author={Vatsal, V.},
	author={Yazdani, S.},
	title={Ternary {D}iophantine equations of signature {$(p,p,3)$}},
	date={2004},
	ISSN={0010-437X},
	journal={$Compos. ~Math.$},
	volume={140},
	number={6},
	pages={1399\ndash 1416},
	url={https://doi.org/10.1112/S0010437X04000983},
}

\bib{MR2455707}{article}{
	author={Bremner, A.},
	title={On the equation {$Y^2=X^5+k$}},
	date={2008},
	ISSN={1058-6458},
	journal={$Experiment.~ Math.$},
	volume={17},
	number={3},
	pages={371\ndash 374},
	url={http://projecteuclid.org/euclid.em/1227121389},
}

\bib{bremner_magma}{article}{
  author = {Bremner, A.},
  title = {Hyperelliptic Curve of Genus 2 MAGMA code},
  date = {2021},
  journal = { available on \href{https://github.com/Chatchawan-git/Hyperelliptic-Curve-of-Genus-2-MAGMA.git}{github.com/Chatchawan-git/Hyperelliptic-Curve-of-Genus-2-MAGMA.git}},
}
\bib{MR1850605}{incollection}{
	author={Bruin, N.},
	title={On powers as sums of two cubes},
	date={2000},
	booktitle={$Algorithmic~ number~ theory$ ({L}eiden, 2000)},
	series={$Lecture ~Notes ~in ~Comput. ~Sci.$},
	volume={1838},
	publisher={Springer, Berlin},
	pages={169\ndash 184},
	url={https://doi.org/10.1007/107220289},
}
\bib{ChimShoreyNair}{article}{
	author={Chim, K. C.},
	author={Shorey, T. N.},
	author={Nair, S. G.},
	title={Explicit $abc$-conjecture and Its applications},
	date={2018},
	journal={$Hardy-Ramanujan ~Journal$},
	volume={41},
	pages={143\ndash 156},
}
\bib{ChimShoreySinha}{article}{
author = {Chim, Kwok Chi},
author = {Shorey, T. N.},
author = { Sinha, S. B.},
title = {On Baker's explicit abc-conjecture},
journal = {$Publ. ~Math. ~Debrecen$},
date = {2019},
volume = {94},
number = {3-4},
pages = {435-453},
ISSN = {0033-3883},
}
\bib{MR1468926}{article}{
	author={Darmon, H.},
	author={Merel, L.},
	title={Winding quotients and some variants of {F}ermat's last theorem},
	date={1997},
	ISSN={0075-4102},
	journal={$J. ~Reine ~Angew. ~Math.$},
	volume={490},
	pages={81\ndash 100},
}

\bib{DoylePoonen}{article}{
author = {Doyle, J. R.}
author = {Poonen, B.},
title = {Gonality of dynatomic curves and strong uniform boundedness of
   preperiodic points},
journal = {$Compos.~ Math.$},
date = {2020},
volume = {156},
number = {4},
pages = {733-743},
ISSN = {0010-437X},
}
\bib{MR2075481}{article}{
	author={Ellenberg, J.~S.},
	title={Galois representations attached to q-curves and the generalized
		{F}ermat equation {$A^4+B^2=C^p$}},
	date={2004},
	ISSN={0002-9327},
	journal={$Amer. ~J. ~Math.$},
	volume={126},
	number={4},
	pages={763\ndash 787},
	url={http://muse.jhu.edu/journals/american_
		journal_of_mathematics/v126/126.4ellenberg.pdf},
}

\bib{MR1480542}{article}{
	author={Flynn, E.~V.},
	author={Poonen, B.},
	author={Schaefer, E.~F.},
	title={Cycles of quadratic polynomials and rational points on a
		genus-{$2$} curve},
	date={1997},
	ISSN={0012-7094},
	journal={$Duke~ Math. ~J.$},
	volume={90},
	number={3},
	pages={435\ndash 463},
	url={https://doi.org/10.1215/S0012-7094-97-09011-6},
}

\bib{MR3493372}{article}{
	author={Freitas, N.},
	title={On the {F}ermat-type equation {$x^3+y^3=z^p$}},
	date={2016},
	ISSN={0010-2571},
	journal={$Comment. ~Math. ~Helv.$},
	volume={91},
	number={2},
	pages={295\ndash 304},
	url={https://doi.org/10.4171/CMH/386},
}

\bib{MR3266961}{article}{
	author={Hutz, B.},
	title={Determination of all rational preperiodic points for morphisms of
		{PN}},
	date={2015},
	ISSN={0025-5718},
	journal={$Math. ~Comp.$},
	volume={84},
	number={291},
	pages={289\ndash 308},
	url={https://doi.org/10.1090/S0025-5718-2014-02850-0},
}

\bib{MR2997579}{article}{
	author={Laishram, S.},
	author={Shorey, T.~N.},
	title={Baker's explicit {$abc$}-conjecture and applications},
	date={2012},
	ISSN={0065-1036},
	journal={$Acta ~Arith.$},
	volume={155},
	number={4},
	pages={419\ndash 429},
	url={https://doi.org/10.4064/aa155-4-6},
}
\bib{Looper1}{article}{
	author={Looper, N.~R. },
	title={The Uniform Boundedness and Dynamical Lang Conjectures for polynomials},
	journal={Preprint available at 	\href{https://arxiv.org/abs/2105.05240}{arXiv:2105.05240}},
	url={https://arxiv.org/abs/2105.05240},
}
\bib{Looper}{article}{
	author = {Looper, N.~R. },
     title = {Dynamical uniform boundedness and the {$abc$}-conjecture},
   journal = {Invent. Math.},
    volume = {225},
      year = {2021},
      number = {1},
     pages = {1--44},
      ISSN = {0020-9910},
}
\bib{MR754559}{book}{
	author={Mason, R.~C.},
	title={$Diophantine ~equations  ~over ~function ~fields$},
	series={London Mathematical Society Lecture Note Series},
	publisher={Cambridge University Press, Cambridge},
	date={1984},
	volume={96},
	ISBN={0-521-26983-0},
	url={https://doi.org/10.1017/CBO9780511752490},
}

\bib{MR1065152}{article}{
	author={Masser, D.~W.},
	title={Note on a conjecture of {S}zpiro},
	date={1990},
	ISSN={0303-1179},
	journal={$Ast\'erisque$},
	number={183},
	pages={19\ndash 23},
	note={S\'eminaire sur les Pinceaux de Courbes Elliptiques (Paris,
		1988)},
}

\bib{MR488287}{article}{
	author={Mazur, B.},
	title={Modular curves and the {E}isenstein ideal},
	date={1977},
	ISSN={0073-8301},
	journal={$Inst. ~Hautes ~\'{E}tudes ~Sci. ~Publ. ~Math.$},
	number={47},
	pages={33\ndash 186 (1978)},
	url={http://www.numdam.org/item?id=PMIHES_1977__47__33_0},
	note={With an appendix by Mazur and M. Rapoport},
}

\bib{MR1369424}{article}{
	author={Merel, L.},
	title={Bornes pour la torsion des courbes elliptiques sur les corps de
		nombres},
	date={1996},
	ISSN={0020-9910},
	journal={$Invent. ~Math.$},
	volume={124},
	number={1-3},
	pages={437\ndash 449},
	url={https://doi.org/10.1007/s002220050059},
}

\bib{MR1665198}{article}{
	author={Morton, P.},
	title={Arithmetic properties of periodic points of quadratic maps.
		{II}},
	date={1998},
	ISSN={0065-1036},
	journal={$Acta ~Arith.$},
	volume={87},
	number={2},
	pages={89\ndash 102},
	url={https://doi.org/10.4064/aa-87-2-89-102},
}

\bib{MR1264933}{article}{
	author={Morton, P.},
	author={Silverman, J.~H.},
	title={Rational periodic points of rational functions},
	date={1994},
	ISSN={1073-7928},
	journal={$Internat. ~Math. ~Res. ~Notices$},
	number={2},
	pages={97\ndash 110},
	url={https://doi.org/10.1155/S1073792894000127},
}

\bib{MR3447581}{article}{
	author={Narkiewicz, W.},
	title={On a class of monic binomials},
	date={2013},
	ISSN={0081-5438},
	journal={$Proc. ~Steklov ~Inst. ~Math.$},
	volume={280},
	number={suppl. 2},
	pages={S65\ndash S70},
	url={https://doi.org/10.1134/S0081543813030073},
}
\bib{MR2674537}{article}{
      author={Narkiewicz, W.},
       title={Cycle-lengths of a class of monic binomials},
        date={2010},
        ISSN={0208-6573},
     journal={$Funct. ~Approx. ~Comment. ~Math.$},
      volume={42},
      number={part 2},
       pages={163\ndash 168},
         url={https://doi.org/10.7169/facm/1277811639},
}
\bib{MR992208}{article}{
	author={Oesterl\'e, J.},
	title={Nouvelles approches du ``th\'eor\`eme'' de {F}ermat},
	date={1988},
	ISSN={0303-1179},
	journal={$Ast\'erisque$},
	number={161-162},
	pages={Exp.\ No.\ 694, 4, 165\ndash 186 (1989)},
	note={S\'eminaire Bourbaki, Vol. 1987/88},
}

\bib{lte}{article}{
  author = {Parvardi, A. H.},
  title = {Lifting The Exponent Lemma (LTE)},
  date = {2011},
}

\bib{MR1655978}{article}{
	author={Poonen, B.},
	title={Some {D}iophantine equations of the form {$x^n+y^n=z^m$}},
	date={1998},
	ISSN={0065-1036},
	journal={$Acta ~Arith.$},
	volume={86},
	number={3},
	pages={193\ndash 205},
	url={https://doi.org/10.4064/aa-86-3-193-205},
}
\bib{PoonenClass}{article}{
	author={Poonen, B.},
	title={. The classification of rational preperiodic points of quadratic polynomials over $\mathbb{Q}$: a refined conjecture},
	date={1998},
	ISSN={0025-5874},
	journal={$Math ~Z.$},
	volume={228},
	pages={11\ndash 29},
	url={https://doi.org/10.1007/PL00004405},
}
\bib{Sadek}{article}{
	author={Sadek, M.},
	title={On rational periodic points of $x^d+c$},
	journal={Preprint available on 	\href{https://arxiv.org/abs/1804.09839}{arXiv:1804.09839}},
	url={https://arxiv.org/abs/1804.09839},
}

\bib{MR2316407}{book}{
	author={Silverman, J.~H.},
	title={The arithmetic of dynamical systems},
	series={$Graduate~ Texts ~in~ Mathematics$},
	publisher={Springer, New York},
	date={2007},
	volume={241},
	ISBN={978-0-387-69903-5},
	url={https://doi.org/10.1007/978-0-387-69904-2},
}

\bib{MR2465796}{article}{
	author={Stoll, M.},
	title={Rational 6-cycles under iteration of quadratic polynomials},
	date={2008},
	ISSN={1461-1570},
	journal={$LMS~ J. ~Comput. ~Math.$},
	volume={11},
	pages={367\ndash 380},
	url={https://doi.org/10.1112/S1461157000000644},
}

\bib{MR1270956}{article}{
	author={Walde, R.},
	author={Russo, P.},
	title={Rational periodic points of the quadratic function
		{$Q_c(x)=x^2+c$}},
	date={1994},
	ISSN={0002-9890},
	journal={$ Amer. ~Math. ~Monthly$},
	volume={101},
	number={4},
	pages={318\ndash 331},
	url={https://doi.org/10.2307/2975624},
}
\end{thebibliography}

\end{document}